\documentclass[12pt]{amsart}
\usepackage{amssymb}
\usepackage{amsthm}
\usepackage{amsmath}
\usepackage[curve,matrix,arrow]{xy}
\usepackage{fullpage}

\theoremstyle{plain}\newtheorem{Theorem}{Theorem}[section]
\theoremstyle{plain}
\theoremstyle{plain}\newtheorem{Corollary}[Theorem]{Corollary}
\theoremstyle{plain}\newtheorem{Lemma}[Theorem]{Lemma}
\theoremstyle{plain}\newtheorem{Proposition}[Theorem]{Proposition}
\theoremstyle{definition}
\theoremstyle{definition}
\theoremstyle{definition}
\theoremstyle{definition}\newtheorem{Remark}[Theorem]{Remark}
\theoremstyle{plain}

\newcommand{\ooplus}{\mathop{\text{\Large$\oplus$}}}
\newcommand{\HH}{H\!H}
\newcommand{\half}{{\textstyle\frac{1}{2}}}

    \def\OG{{\mathcal{O}G}}  
      
    \def\OP{{\mathcal{O}P}}
    \def\OQ{{\mathcal{O}Q}}

\def\CH{{\mathcal{H}}}

\def\CL{{\mathcal{L}}}

\def\CO{{\mathcal{O}}}

\def\R{{\mathbb R}}
\def\Z{{\mathbb Z}}

\def\ad{\mathrm{ad}}
           \def\tenk{\otimes_k}     
             \def\ten{\otimes}
\def\chr{\mathrm{char}}

\def\Der{\mathrm{Der}}
\def\dim{\mathrm{dim}}

\def\Ext{\mathrm{Ext}}

\def\Hom{\mathrm{Hom}}           

\def\ker{\mathrm{ker}}           
             
\def\IDer{\mathrm{IDer}}
\def\Im{\mathrm{Im}}

           \def\tenO{\otimes_{\mathcal{O}}}

\def\op{\mathrm{op}}

\def\rk{\mathrm{rk}}         
\def\soc{\mathrm{soc}}

\title[On blocks of defect two and one simple module]
{On blocks of defect two and one simple module, and Lie algebra 
structure of $\HH^1$} 
\author{D. J. Benson, Radha Kessar, and Markus Linckelmann} 
\date{\today}

\begin{document}

\maketitle

\begin{abstract}    
Let $k$ be a field of odd prime characteristic $p$.  We 
calculate the Lie algebra structure of the first Hochschild cohomology
of a class of quantum complete intersections over $k$. As a 
consequence, we prove that if $B$ is a  defect $2$-block of a finite 
group algebra $kG$ whose Brauer correspondent $C$ has a unique 
isomorphism class  of simple modules, then a basic 
algebra of $B$ is a local algebra which can be generated by at most $2\sqrt I$ elements, where 
$I$ is the inertial index of $B$, and where we assume that $k$ is
a splitting field for $B$ and $C$.
\end{abstract}

\section{Introduction}

The purpose of this paper is to examine certain algebras of dimension
$p^2$ over a field of odd characteristic $p$, which occur as the basic
algebras of blocks of finite groups with normal defect groups of order
$p^2$ and a unique
simple module. The goal is to understand the
Brauer correspondents of such blocks. To this end, we make a
detailed examination of the degree one Hochschild cohomology
as a Lie algebra. 

\begin{Theorem} \label{HHoneLiestructure}
Let $k$ be a field of odd prime characteristic $p$ and let
$q\in$ $k^\times$ be an element of finite order $e$ such that
$e\geq$ $2$ and such that $e$ divides $p-1$. Let
$$A = k\langle x,y\ |\ x^p=0=y^p, \ yx=qxy\rangle\ .$$
Set $\CL=$ $\HH^1(A)$ and let $\CL'$ be the derived Lie subalgebra 
of the  Lie algebra $\CL$. Denote by $\soc_{Z(A)}(\CL)$
the socle of $\CL$ as a left $Z(A)$-module. Then $A$ is a split
local symmetric $k$-algebra of dimension $p^2$, and 
the following hold.
\begin{enumerate}
\item[\rm (i)]
We have $\dim_k(\CL)=$ $2(p+(\frac{p-1}{e})^2)$.
\item[\rm (ii)]
We have $Z(\CL)=\{0\}$.
\item[\rm (iii)] 
There is a $2$-dimensional maximal toral subalgebra 
$\CH$ of $\CL$ such that $\CL= \CH\oplus \CL'$.
\item[\rm (iv)]
The derived subalgebra $\CL'$ is nilpotent; in 
particular, $\CL$ is solvable. 
\item[\rm (v)]
We have $\dim_k(\soc_{Z(A)}(\CL)) = 2e$ and  
$\soc_{Z(A)}(\CL) \subseteq$ $Z(\CL')$.
\item[\rm (vi)]
We have $J(Z(A))\CL=$ $\CL'$ and $\dim_k(\CL/\CL')=$ $2$.
\item[\rm (vii)]
We have $\dim_k(Z(\CL'))=2e+2$. In particular, $\CL'$ is
abelian if and only if $e=p-1$.
\item[\rm (viii)]
The subalgebra $\CH$ is $p$-toral, and we have $(\CL')^{[p]}=$ 
$\{0\}$.
\end{enumerate}
\end{Theorem}

See Section \ref{LieStructure} for the proof.
Other papers examining Hochschild cohomology of similar algebras 
include Bergh and Erdmann \cite{BeErd} and
Oppermann \cite{Opp}, but their results and goals lie in different
directions. For example, in \cite{BeErd} it is assumed that $q$ is not
a root of unity. The last statement in Theorem \ref{HHoneLiestructure}
regarding the $p$-restricted structure of $\CL$ is motivated by
invariance results of $p$-power maps in Hochschild cohomology under
derived and stable equivalences in work of Zimmermann \cite{Zimm} and 
Rubio y Degrassi \cite{Rubio}.

To exploit Theorem \ref{HHoneLiestructure} we prove the following 
general theorem, which provides an upper bound for the number of loops
in the quiver of a symmetric split algebra over an arbitrary field.

\begin{Theorem} \label{soclederivationI}
Let $k$ be a field and let $A$ be a symmetric split $k$-algebra. We have
$$\sum_S\ \dim_k(\Ext^1_A(S,S)) \leq \dim_k(\soc_{Z(A)}(\HH^1(A)))$$
where in the sum $S$ runs over a set of representatives of
the isomorphism classes of simple $A$-modules. In particular,
if $A$ is a symmetric split local $k$-algebra, then
$$\dim_k(J(A)/J(A)^2) \leq \dim_k(\soc_{Z(A)}(\HH^1(A)))\ .$$
\end{Theorem}

This will be proved in Theorem \ref{soclederivation} and
Corollary \ref{soclederivationlocal}. Combining the two theorems above
with standard properties of stable equivalences of Morita
type yields the following consequence.

\begin{Corollary} \label{radicaldimCor}
Let $A$ be as in Theorem \ref{HHoneLiestructure}, and let $B$ be a 
split local symmetric $k$-algebra such that there is a stable 
equivalence of Morita type between $A$ and $B$. We have
$$\dim_k(J(B)/J(B)^2)\leq 2e\ .$$
\end{Corollary}

The motivation for the above results comes from local-global 
considerations in the modular representation theory of finite groups. 
Let $G$ be a finite group  and let $B$  be a block of  the group 
algebra $kG$ of $G$ over a field $k$ of odd characteristic. Let $P$ be 
a defect group of $B$, $C$ the  
block of of $kN_G(P)$ in Brauer correspondence with $B$ and let $I$ be 
the inertial index of $B$. Suppose that $k$ is a splitting field for
$B$ and $C$. If $P$ has order $p^2$, then it is known     
that there is a stable equivalence of Morita type between $B$ and $C$.  
If in addition $C$ has a unique isomorphism class of simple modules, 
then $C$ is a matrix algebra over a quantum complete intersection as 
in Corollary~\ref{radicaldimCor}. Moreover, in this case 
$e \leq \sqrt I $ and if $I>1 $, then $e > 1$. Thus, Corollary  
\ref{radicaldimCor} yields the following local-global  result. 

\begin{Corollary} \label{blockradicaldimCor} 
Let $G$ be a finite group  and let $B$ be a block of the group algebra 
$kG$ of $G$ over a field $k$ of odd characteristic $p$. Let $P$ be 
defect group of $B$, $C$ the  block 
of $kN_G(P)$ in Brauer correspondence with $B$ and let $I$ be the 
inertial index of $B$. Suppose that $P$  has  order 
$p^2$, that $C$ has a unique isomorphism class of simple modules, 
and that $k$ is a splitting field for $B$ and $C$. Then $B$ has a
unique isomorphism class of simple modules, and
$$\dim_k(J(B)/J(B)^2)\leq 2\sqrt I\ .$$
\end{Corollary}
 
Corollaries \ref{radicaldimCor} and \ref{blockradicaldimCor}
are proved at the end of Section \ref{LieStructure}.
We note that in the situation of Corollary
\ref{blockradicaldimCor}, 
Brou\'e's abelian defect group conjecture 
\cite{BroueAb} would imply 
that the blocks $B$ 
and $C$ are derived equivalent, and therefore by a result of
Roggenkamp and Zimmermann \cite[Proposition 6.7.4]{ZimBook}, that $B$ and $C$ are
Morita equivalent. Hence, it would follow that the dimension of
$J(B)/J(B)^2$ is two. If $p=3 $, it is known that $B$ and $C$ are
Morita equivalent in this situation \cite{Ke12}.

If $e=2$, then the algebra $A$ in Theorem \ref{HHoneLiestructure} is
Morita equivalent to the nonprincipal block algebra of the finite
group algebra $kG$, where $G=$ $(C_p\times C_p)\rtimes Q_8$,
with $Z(Q_8)$ acting trivially on $C_p\times C_p$, such that
the induced action of $Q_8/Z(Q_8)\cong$ $C_2\times C_2$ is given
by each copy of $C_2$ acting by inversion on the corresponding
copy of $C_p$. Thus $A$ lifts to an $\CO$-free $\CO$-algebra
$\hat A$
which is Morita equivalent to the nonprincipal block
$B_1$ of $\OG$. Here $\CO$ is a complete discrete valuation ring
of characteristic zero with residue field $k$ of odd prime 
characteristic $p$; we assume that $\CO$ contains a primitive $4p$-th 
root of unity.  This algebra $\hat A$ can be described, using the
normalised polynomials $f_n(u)=2T_n(\frac{u}{2})$ of the
Chebyshev polynomials of the first kind $T_n$ (see \S
\ref{liftingSection} for a more detailed review of the notation).

\begin{Theorem} \label{liftO}
With the notation above, the $\CO$-algebra
\[ \hat A = \CO\langle \gamma, \delta \ |\ \gamma\delta+\delta\gamma=0,
\ f_p(\gamma)=0=f_p(\delta)\rangle \]
is a basic algebra of $B_1$. In particular, we have
$k\tenO \hat A\cong$ $A$.
\end{Theorem}

This will be proved in \S \ref{liftingSection}. If $e>2$, it turns out
that it is much harder to describe $\hat A$.
 
\section{Basic background facts}

Let $k$ be a field. For $A$ a finite-dimensional $k$-algebra, we 
denote by $\ell(A)$ the number of isomorphism classes of simple
$A$-modules. We write $A^e=$ $A\tenk A^\op$. We consider $A^e$-modules
as $A$-$A$-bimodules and vice versa, whenever convenient. We denote 
by  $[A,A]$ the additive commutator space, spanned by 
the set of elements $[a,b]=$ $ab-ba$, with $a$, $b\in$ $A$. 
If $A$ is split local, then every element in $A$ is of the form
$\lambda\cdot 1 + r$ for some $\lambda\in$ $k$ and some $r\in$
$J(A)$. This yields immediately the following well-known fact:

\begin{Lemma} \label{commradtwo}
Let $A$ be a finite-dimensional split local $k$-algebra.
We have $[A,A]\subseteq$ $J(A)^2$.
\end{Lemma}

A $k$-algebra $A$ is {\it symmetric} if $A$ is isomorphic
to its $k$-dual $A^\vee$ as an $A$-$A$-bimodule (this implies that
$A$ is finite-dimensional). If $A$ is symmetric, then the socle
of $A$ as a left $A$-module and as a right $A$-module coincide.
If $A$ is also split, then this coincides with the socle of $A$
as an $A$-$A$-bimodule. The image $s\in$ $A^\vee$
of $1_A\in$ $A$ under an $A$-$A$-bimodule isomorphism $A\cong$ $A^\vee$
is called a {\it symmetrising form}. Note that it satisfies
$s(ab)=s(ba)$. 
If $A$ is symmetric with a fixed
choice of a symmetrising form $s$, for any subspace $U$ of $A$ we denote
by $U^\perp$ the subspace consisting of all $a\in$ $A$ satisfying
$s(au)=$ $0$ for all $u\in$ $U$. We have $\dim_k(U)+\dim_k(U^\perp)=$
$\dim_k(A)$, and hence $U^{\perp\perp}=$ $U$.
It is well-known that $[A,A]^\perp=$ $Z(A)$
and that $\soc(A)^\perp=$ $J(A)$. The space $[A,A]$ is contained
in any symmetrising form of $A$. If $A$ is split local 
symmetric, then $\soc(A)$ has dimension $1$ and is the unique minimal 
ideal in $A$; thus, in that case, we have $[A,A]\cap \soc(A)=$ $\{0\}$.

\begin{Lemma} \label{soctwocentral}
Let $A$ be a split local symmetric $k$-algebra. Then
$\soc^2(A)\subseteq$ $Z(A)$.
\end{Lemma}

\begin{proof}
Choose a symmetrising form of $A$. The statement follows from 
Lemma \ref{commradtwo}, since $(J(A)^2)^\perp=\soc^2(A)$ and
$[A,A]^\perp=Z(A)$.
\end{proof}

For $A$ a split finite-dimensional $k$-algebra, the semisimple
quotient $A/J(A)$ is a direct product of matrix algebras, hence
symmetric. Thus $(A/J(A))^\vee\cong A/J(A)$ as $A$-$A$-bimodules.
Moreover, we have an $A$-$A$-bimodule isomorphism
$A/J(A)\cong \ooplus_S\ S\tenk S^\vee$, where $S$ runs over a set
of representatives of the isomorphism classes of simple $A$-modules.
If $A$ is split and symmetric, then $A/J(A)\cong\soc(A)$ 
and $(A/\soc(A))^\vee\cong J(A)$ as $A$-$A$-bimodules. 

\begin{Lemma}[{\cite[Chapter IX, Corollary 4.4]{CE}}] 
\label{HHExtlemma1}
Let $A$ be a finite-dimensional $k$-algebra, and let $S$,
$T$ be finite-dimensional $A$-modules. There is a canonical graded 
$k$-linear isomorphism 
\[ \HH^*(A;S\tenk T^\vee)\cong\Ext_A^*(T,S). \]
\end{Lemma}

\begin{proof}
A standard adjunction, with $T$ viewed as an $A$-$k$-bimodule, 
yields for any projective $A^e$-module $P$ a
natural isomorphism 
\[ \Hom_A(P\ten T, S) \cong\Hom_{A^e}(P, \Hom_k(T,S))\cong
\Hom_{A^e}(P, S\tenk T^\vee). \]  
By naturality, 
replacing $P$ by a projective resolution of $A$ as an
$A^e$-module yields an isomorphism of cochain complexes.
Taking cohomology yields the statement.
\end{proof}

\begin{Lemma} \label{HHExtlemma2}
Let $A$ be a split symmetric $k$-algebra. We have a graded
$k$-linear isomorphism 
\[ \HH^*(A;\soc(A))\cong \ooplus_S \Ext^*_A(S,S)\ , \]
where $S$ runs over a set of representatives of the isomorphism
classes of simple $A$-modules. 
In particular, we have
$$\dim_k(\Hom_{A^e}(A,\soc(A)))=\ell(A)\ ,$$
$$\dim_k(\HH^1(A;\soc(A)))=\sum_S\dim_k(\Ext^1_A(S,S))\ ,$$
where $S$ runs over a set of representatives of the isomorphism
classes of simple $A$-modules. 
\end{Lemma}

\begin{proof}
As mentioned above, we have $A$-$A$-bimodule isomorphisms
\[ \soc(A)\cong A/J(A)\cong  \ooplus_S\ S\tenk S^\vee, \] 
where $S$ runs over a set of representatives of the isomorphism classes of 
simple $A$-modules. Thus the isomorphism follows from the
previous lemma. Comparing dimensions in degree $0$ and in degree $1$
yields the two equalities.
\end{proof}

\section{Calculating derivations on symmetric algebras}

Let $k$ be a field and let $A$ be a finite-dimensional $k$-algebra.
We will use the description of $\HH^1(A)$ as outer derivations. 
A $k$-linear map $f \colon A\to$ $A$ is a {\it derivation} if 
$f(ab)=$ $af(b)+f(a)b$ for all $a$, $b\in$ $A$. If $z\in$ $Z(A)$
and $f$ is a derivation on $A$, then $z\cdot f$ defined by $(z\cdot f)(a)=$
$zf(a)$ is a derivation on $A$. In this way, the set of derivations
$\Der(A)$ on $A$ becomes a $Z(A)$-module. 
If $x\in$ $A$, then the map $[x,-]$
sending $a\in$ $A$ to $[x,a]=$ $xa-ax$ is a derivation; any
derivation of this form is called an {\it inner derivation},
of $A$, and the set $\IDer(A)$ of inner derivations of
$A$ is a $Z(A)$-submodule of $\Der(A)$. We have a canonical
isomorphism $\HH^1(A)\cong$ $\Der(A)/\IDer(A)$; see e.g.
\cite[9.2.1]{Weibel}. The $\HH^0(A)$-module structure and
the $Z(A)$-module structure on $\Der(A)/\IDer(A)$ correspond
to each other through the canonical isomorphism $\HH^0(A)\cong$
$Z(A)$. Any derivation $f$ on $A$ satisfies 
$f(1)=$ $0$, since $f(1)=$ $f(1\cdot 1)=$ $f(1)\cdot 1+ 1\cdot f(1)=$ 
$2 f(1)$, hence $\ker(f)$ is a unitary subalgebra of $A$. 
The space $\IDer(A)$ is isomorphic
to the quotient of $A$ by the kernel of the map $x\mapsto$ $[x,-]$,
hence $\dim_k(\IDer(A))=$ $\dim_k(A)-\dim_k(Z(A))$.
Thus if $A$ is symmetric, then  $\dim_k(\IDer(A))=$ $\dim_k([A,A])$.
For any $Z(A)$-module $H$ we denote by $\soc_{Z(A)}(H)$ its 
socle as a $Z(A)$-module.

\begin{Theorem} \label{soclederivation}
Let $A$ be a symmetric split $k$-algebra and let $E$ be a
maximal semisimple subalgebra.
Let $f \colon A\to A$ be an $E$-$E$-bimodule homomorphism satisfying 
$E+J(A)^2\subseteq$ $\ker(f)$ and $\Im(f)\subseteq$ $\soc(A)$. 
Then $f$ is a derivation on $A$ in $\soc_{Z(A)}(\Der(A))$, and if 
$f\neq$ $0$, then $f$ is an outer derivation of $A$.
In particular, we have
$$\sum_S\ \dim_k(\Ext^1_A(S,S)) \leq \dim_k(\soc_{Z(A)}(\HH^1(A)))$$
where in the sum $S$ runs over a set of representatives of
the isomorphism classes of simple $A$-modules.
\end{Theorem}

\begin{proof}
Let $a$, $b\in$ $A$. By the Wedderburn--Malcev theorem, we have 
$A=$ $E\oplus J(A)$. Thus $a=$ $c+r$ and $b=$ $d+s$ for some
$c$, $d\in$ $E$ and $r$, $s\in$ $J(A)$. The hypotheses on
$f$ imply that 
\[ f(ab)=f(cd+cs+rd+rs)=f(cs+rd)=cf(s)+f(r)d=af(b)+f(a)b. \] 
This shows that $f$ is a
derivation. Suppose that $f$ is an inner derivation. Then
$\Im(f)\subseteq$ $[A,A]\cap \soc(A)$. But $\Im(f)$ is also 
an $E$-$E$-bimodule. Any $E$-$E$-bimodule contained in 
$\soc(A)$ is in fact an ideal. The space $[A,A]$ contains 
no nonzero ideal as $A$ is symmetric. Thus $f$ is either 
zero or an outer derivation. Since $J(Z(A))$ is contained
in $J(A)$, any such derivation is annihilated by $J(Z(A))$.
This shows that
$\Hom_{E^e}(J(A)/J(A)^2,\soc(A))$ is isomorphic
to a subspace of $\soc_{Z(A)}(\HH^1(A))$. Since $J(A)$ 
annihilates $J(A)/J(A)^2$ and $\soc(A)$, this subspace is 
isomorphic to $\Hom_{A^e}(J(A)/J(A)^2,\soc(A))$.
As $A$ is symmetric, we have 
\[ \soc(A)\cong A/J(A)\cong\ooplus_S\ S\tenk S^\vee, \] 
with $S$ running over a set of 
representatives of the isomorphism classes of simple 
$A$-modules. The dimension of 
$\Hom_{A^e}(J(A)/J(A)^2,S\tenk S^\vee)$ is equal to
the number of summands of the $A^e$-module $J(A)/J(A)^2$ 
isomorphic to $S\tenk S^\vee$. If $i$ is a primitive 
idempotent such that $iS\neq$ $\{0\}$, then $S$ is the 
unique simple quotient of $Ai$, hence $S^\vee$ is the 
unique simple quotient of $iA$, and thus $S^\vee i$ is 
one-dimensional. It then follows that the dimension of 
$\Hom_{A^e}(J(A)/J(A)^2,S\tenk S^\vee)$ is equal to
the number of summands of $J(A)i/J(A)^2i$ isomorphic to $S$,
and that is precisely $\dim_k(\Ext^1_A(S,S))$.
\end{proof}

\begin{Corollary} \label{soclederivationlocal}
Let $A$ be a split local symmetric $k$-algebra. 
Let $f \colon A\to A$ be a $k$-linear map satisfying 
$1+J(A)^2\subseteq$ $\ker(f)$ and $\Im(f)\subseteq$ $\soc(A)$. 
Then $f$ is a derivation on $A$ in $\soc_{Z(A)}(\Der(A))$, and if 
$f\neq$ $0$, then $f$ is an outer derivation of $A$.
In particular, we have
$$\dim_k(J(A)/J(A)^2) \leq \dim_k(\soc_{Z(A)}(\HH^1(A)))\ .$$
\end{Corollary}

\begin{proof}
Since $A$ is split local, we have $\dim_k(J(A)/J(A)^2)=$
$\dim_k(\Ext_A^1(S,S))$, where $S=$ $A/J(A)$ is the unique
simple $A$-module, up to isomorphism. Moreover, $k\cdot 1_k$ is
the unique maximal semisimple subalgebra of $A$. The result
follows from Theorem \ref{soclederivation}.
\end{proof}

Combining Theorem \ref{soclederivation} and 
Corollary \ref{soclederivationlocal}
implies Theorem \ref{soclederivationI}.

\begin{Remark} \label{soclederivationRemark}
Note that $\HH^1(A)$ is annihilated by the projective ideal
$Z^{pr}(A)$ in $Z(A)$, hence $\HH^1(A)$ is a module over the 
stable center $\bar Z(A)=$ $Z(A)/Z^{pr}(A)$, and 
\[ \soc_{Z(A)}(\HH^1(A))=\soc_{\bar Z(A)}(\HH^1(A)). \] 
This shows that $\soc_{Z(A)}(\HH^1(A))$ is invariant under stable 
equivalences of Morita type.
\end{Remark}

It is possible to give a more structural proof of the inequality in
Theorem \ref{soclederivation}, based on the following result.

\begin{Proposition} \label{AsocA}
Let $A$ be a split symmetric $k$-algebra. We have canonical short 
exact sequences 
\begin{gather*}
\xymatrix{0\ar[r]&\Hom_{A^e}(A,\soc(A))\ar[r] &
\Hom_{A^e}(A,A)\ar[r] & \Hom_{A^e}(A,A/\soc(A))
\ar[r] & 0} \\
\xymatrix{0\ar[r]&\Hom_{A^e}(A/J(A),A)\ar[r] &
\Hom_{A^e}(A,A)\ar[r] & \Hom_{A^e}(J(A),A)
\ar[r] & 0}
\end{gather*}
In particular, we have
$$\dim_k(Z(A)) - \ell(A) = \dim_k(\Hom_{A^e}(A,A/\soc(A)))
=\dim_k(\Hom_{A^e}(J(A),A))\ .$$
\end{Proposition}

\begin{proof}
We may assume that $A$ is basic. Any $A^e$-homomorphism from $A$ to 
$A/\soc(A)$ is in particular a homomorphism of left $A$-modules. As 
such, it lifts to an endomorphism of $A$, and hence is induced by 
right multiplication with an element $y\in$ $A$, followed by the 
canonical map $A\to$ $A/\soc(A)$. For this to induce a bimodule 
homomorphism from $A$ to $A/\soc(A)$ a necessary condition is 
$[y,a]\in$ $\soc(A)$ for all $a\in$ $A$. Using that $A$ is basic, one 
can show that this forces $y\in$ $Z(A)$. Indeed, for
any primitive idempotent $i$ we have 
\[ [y,i]=yi-iy=yi-iyi-iy(1-i)\in\soc(A). \] Since $A$ is basic, this
forces $iy(1-i)=$ $0$ because $\soc(A(1-i))$ has no
submodule isomorphic to the simple module $Ai/J(A)i$. Thus $iy=$ 
$iyi$, and a similar argument shows $iyi=$ $yi$. Thus $y$ commutes
with all primitive idempotents. For $a\in$ $Ai$ we have
$[y,a]=$ $yai-ayi\in$ $\soc(Ai)$, so this is annihilated
by $1-i$, hence equal to $iyiai-iaiyi\in$ $\soc(iAi)\cap [iAi,iAi]$,
which is zero because the local algebra $iAi$ is symmetric.
Thus $y\in$ $Z(A)$, which means precisely that the induced
homomorphism $A\to$ $A/\soc(A)$ lifts to a bimodule 
homomorphism $A\to$ $A$, whence the exactness of the first sequence
as stated. The second sequence is obtained from applying duality to 
the first.
By Lemma \ref{HHExtlemma2}, the dimension of the left term
in the first sequence is $\ell(A)$, and the middle term is isomorphic 
to $Z(A)$, which proves the first equality. The second equality is
obtained via duality. 
\end{proof} 

\begin{Remark}
The inequality in Theorem \ref{soclederivation} can be proved using
Proposition \ref{AsocA} as follows.
We consider the long exact sequence obtained from applying the
functor $\Hom_{A^e}(A,-)$ to the short
exact sequence of $A^e$-modules
$$\xymatrix{0\ar[r] &  \soc(A)\ar[r] &  A\ar[r] &
A/\soc(A)\ar[r] & 0}$$
This yields in particular an exact sequence
$$\xymatrix{\Hom_{A^e}(A,A) \ar[r] &
\Hom_{A^e}(A, A/\soc(A)) \ar[r] &
HH^1(A; \soc(A))\ar[r] & HH^1(A) }$$
By Proposition \ref{AsocA} the first map is surjective.
Thus the second map is zero, hence the third map is
injective.
Thus $HH^1(A;\soc(A))$ is isomorphic to a subspace
of $HH^1(A)$. Since $J(Z(A))\subseteq$ $J(A)$, this subspace
is contained in $\soc_{Z(A)}(HH^1(A))$. The inequality in
Theorem \ref{soclederivation} follows from \ref{HHExtlemma2}.

The surjectivity of the first map in the above exact sequence
can be used to give a proof of a result of Brandt \cite{Brandt}, as follows.
Identify $\Hom_{A^e}(J(A)/J(A)^2;A)$ with
a subspace of $\Hom_{A^e}(J(A);A)$ via the canonical surjection
$J(A)\to$ $J(A)/J(A)^2$. If $J(A)^2$ is nonzero, then
$\Hom_A(J(A)/J(A)^2,A)$ is strictly smaller than $\Hom_A(J(A),A)$,
because the inclusion map $J(A)\subseteq$ $A$ does not factor through
$J(A)/J(A)^2$. Since $J(A)/J(A)^2$ is semisimple, we have
\[ \Hom_{A^e}(J(A)/J(A)^2,A)=\Hom_{A^e}(J(A)/J(A)^2,\soc(A)), \] 
which in turn (as
observed in the proof of \ref{soclederivation}) is isomorphic to
$\ooplus_S \Ext^1_A(S,S)$, where $S$ runs over a set of representatives
of the isomorphism classes of simple $A$-modules. Thus, if $A$ is split
symmetric such that $J(A)^2\neq$ $\{0\}$, then the dimension of
$\Hom_{A^e}(J(A),A)$ is strictly greater than that of
$\ooplus_S \Ext^1_A(S,S)$. Proposition \ref{AsocA} implies in that
case the inequality
$$\dim_k(Z(A))-\ell(A) \geq 1 + \sum_S \dim_k(\Ext^1_A(S,S))$$
due to Brandt \cite[Theorem B]{Brandt}.

It follows that the integers $\dim_k(Z(A))-\ell(A)-1$ and
$\dim_k(\soc_{Z(A)}(HH^1(A)))$ are both upper bounds for
$\sum_S \dim_k(\Ext^1_A(S,S))$. These two upper bounds are not
comparable in general, since they arise from unrelated parts
of a long exact sequence with a zero map. The following two examples
illustrate this. Suppose that $k$ has odd prime characteristic $p$.
If $A=$ $k(C_p\rtimes C_{p-1})$, with $C_{p-1}$ acting regularly
on the nontrivial elements of $C_p$, then standard
calculations yield
$$\dim_k(Z(A))-\ell(A)-1= 0< \dim_k(\soc_{Z(A)}(HH^1(A)))=1\ .$$
By contrast, if $A$ is as in Theorem \ref{HHoneLiestructure}, then,
using Lemma \ref{centralbasis} below, we have
$$\dim_k(Z(A))-\ell(A)-1 = \textstyle\left(\frac{p-1}{e}\right)^2+2p-3
\geq \dim_k(\soc_{Z(A)}(HH^1(A)))= 2e\ ,$$
with equality if and only if $e=p-1$.
\end{Remark}

Derivations with image in the second socle layer are characterised 
as follows.

\begin{Proposition} \label{secondsoclederivation}
Let $A$ be a split local symmetric $k$-algebra, let $\{x_1,x_2,..,x_r\}$
be a $k$-basis of a complement of $J(A)^2$ in $J(A)$, and let
$z$ be a nonzero element in $\soc(A)$. There is a basis
$\{y_1,y_2,..,y_r\}$ of a complement of $\soc(A)$ in $\soc^2(A)$
such that $x_iy_i=$ $y_ix_i=$ $z$ for $1\leq$ $i\leq$ $r$, and
such that $x_iy_j=$ $y_jx_i=$ $0$, for $1\leq$ $i,j\leq$ $r$, $i\neq$ 
$j$. Let $f \colon A\to A$ be a $k$-linear map satisfying
$1+J(A)^2\subseteq$ $\ker(f)$, such that 
$$f(x_i)= \sum_{j=1}^r\ \sigma_{i,j} y_j$$
for some coefficients $\sigma_{i,j}\in$ $k$, $1\leq$ $i,j\leq$ $r$.
\begin{enumerate}
\item[\rm (i)]
The map $f$ is a derivation if and only if $\sigma_{i,j}=$
$-\sigma_{j,i}$ for all $i$, $j$, $1\leq$ $i,j\leq$ $r$.
In particular, if $\chr(k)\neq$ $2$ and if $f$ is a
derivation, then $\sigma_{i,i}=$ $0$ for $1\leq$ $i\leq$ $r$, and
the space of derivations obtained in this
way has dimension $\frac{r(r-1)}{2}$.
\item[\rm (ii)]
If $f$ is an inner derivation, then $\Im(f)\subseteq$ 
$\soc(Z(A))\cap \soc^2(A)\cap [A,A]$, and $\Im(f)$
is contained in a complement
of $\soc(A)$ in $\soc(Z(A))\cap\soc^2(A)$.
\end{enumerate}
\end{Proposition}

\begin{proof}
Let $a$, $b\in$ $A$. In the following sums, the indices
$i$ and $j$ run from $1$ to $r$. Write 
\[ a=\sum_i\ \alpha_ix_i + \lambda\cdot 1 + u \]
with coefficients $\alpha_i$ and $\lambda$ in $k$, and $u\in$ $J(A)^2$.
Similarly, write
\[ b=\sum_j\ \beta_ix_i + \mu\cdot 1 + v \]
with $\beta_i,\mu\in k$ and $v\in J(A)^2$.
Thus 
\begin{align*} 
f(a)&= \sum_i\ \alpha_ix_i=
\sum_{i,j}\ \alpha_i\sigma_{i,j}y_j,  \\
f(b)&= \sum_i\ \beta_ix_i=
\sum_{i,j}\ \beta_i\sigma_{i,j}y_j. 
\end{align*}
Since $u$, $v$ annihilate the $y_i$, short calculations,
using the hypotheses on $f$, yield
\begin{align*}
f(a)b&= (\sum_{i,j}\alpha_i\sigma_{i,j}\beta_j)z + \mu f(a), \\
af(b)&=(\sum_{i,j}\beta_i\sigma_{i,j}\alpha_j)z + \lambda f(b), \\
f(ab)&=\mu f(a) + \lambda f(b).
\end{align*}
Thus $f$ is a derivation if and only if 
\[ \sum_{i,j} (\alpha_i\sigma_{i,j}\beta_j +  \beta_i\sigma_{i,j} \alpha_j)=0 \]
for all choices of coefficients $\alpha_i$, $\beta_j$.
This holds if and only if $\sigma_{i,j}=-\sigma_{j,i}$ for
all $i$, $j$. Statement (i) follows.
Suppose now that $f$ is an inner derivation,
say $f= [w,-]$ for some $w\in A$. By the assumptions on
$f$, we have $[w,J(A)^2]= \{0\}$ and $[w,A]\subseteq 
\soc^2(A)\subseteq Z(A)$, where the last inclusion is
from Lemma \ref{soctwocentral}. Note that $\Im(f)$ is spanned
by the $[w,x_i]$, $1\leq i\leq r$. If $c\in$ $J(Z(A))$, then
\[ [w,x_i]c= wx_ic-x_iwc= w(x_ic)-(x_ic)w= 0, \] 
since $x_ic$ is contained in $J(A)^2$, hence commutes with $w$.
Thus $\Im(f)$ is annihilated by $J(Z(A))$, implying 
$\Im(f)\subseteq \soc(Z(A))\cap \soc^2(A)$. Since
also $\Im(f)\subseteq [A,A]$, which intersects $\soc(A)$
trivially as $A$ is symmetric split local, statement
(ii) follows. 
\end{proof}

%

For monomial algebras, the Lie algebra structure of $\HH^1$ has been
calculated in work of Strametz \cite{Strametz}.
Maximal diagonalisable Lie subalgebras of $\HH^1$ have been calculated
by Le Meur \cite{LeMeur} for certain algebras without oriented cycles.
The dimension of $\dim(\HH^1(A))$ is related to combinatorial data of
the quiver of $A$ in work of de la Pe\~na and Saor\'{\i}n \cite{delaPSa}.

\section{The dimension of $\HH^1(A)$}

Let $k$ be a field of odd prime characteristic $p$,  let 
$1\ne q \in k^{\times}$ have order $e$ dividing $p-1$, and let
\[ A = k \langle x,y\ |\ x^p=y^p=0,\ yx=qxy\rangle. \]
Then $A$ is a symmetric local $k$-algebra of dimension $p^2$,
having the set of monomials
$$V = \{x^iy^j\ |\ 0\leq i, j\leq p-1 \}$$ 
as a $k$-basis. The linear map $A\to$ $k$ sending $x^{p-1}y^{p-1}$ to
$1$ and all other monomials in $V$ to $0$ is a symmetrising form
for $A$. 

\begin{Remark}
One can define an algebra $A$ as above for arbitrary $q\in$ $k^\times$, 
but unless the order of $q$ divides $p-1$, this yields a selfinjective 
algebra which is not symmetric. Indeed, if $A$ is symmetric, then any 
symmetrising form $s$ of $A$ is nonzero on the socle element 
$x^{p-1}y^{p-1}$.
Thus 
\[ 0\neq s(x^{p-1}y^{p-1})= s(x^{p-2}y^{p-1}x)=
q^{p-1}s(x^{p-1}y^{p-1}), \] 
and hence $q^{p-1}= 1$. Thus the algebras
arising for $q$ not of order dividing $p-1$ are not Morita equivalent
to block algebras of finite groups.
\end{Remark}

The purpose of this section is to determine the dimension of $\HH^1(A)$.

\begin{Proposition} \label{HHonedim}
We have $\dim(\HH^1(A))= 2(p +  (\frac{p-1}{e}  )^2)$.
\end{Proposition}

We start with some technical observations. The subset 
$$V' =\{x^iy^j\ |\ 0\leq i, j\leq p-1,\ (i,j)\neq(0,0) \}$$ 
of $A$ is a $k$-basis of $J(A)$, and the element $x^{p-1}y^{p-1}$ 
spans $\soc(A)$. For $r\geq $ $0$ the subset
\[ V_r= \{x^iy^j\ |\ 0\leq i, j\leq p-1,\ i+j\geq r\} \] 
of $V$ is a $k$-basis of $J(A)^r$

\begin{Lemma} \label{centralbasis}\ 

\begin{enumerate}
\item[\rm (i)]
The set $\{x^iy^j\ |\ 0\leq i,j\leq p-1, \,\ i\ \mathrm{and}\ j
\ \mathrm{divisible  \ by }  \ e,\ \mathrm{or}\ i=p-1,\ 
\mathrm{or}\ j=p-1\}$ 
is a $k$-basis of $Z(A)$. In particular, we have
$$\dim_k(Z(A))=\textstyle\left(\frac{p-1}{e} \right)^2+2p-1.$$
\item[\rm (ii)]
The set $\{x^{i}y^{p-1}, x^{p-1}y^{j} \ |   
\  p-e \leq i, j \leq p-1     \}$ is
a $k$-basis of $\soc(Z(A))$; in particular, we have 
$\dim_k(\soc(Z(A)))=$ $2e-1$.
\item[\rm (iii)]
We have $J(Z(A))\subseteq$ $J(A)^e$, and  if $e =2 $, then 
$\soc(Z(A))=\soc^2(A) \subseteq J(A)^2$.
\end{enumerate}
\end{Lemma}

\begin{proof}
If $x$ and $y$ commute with a linear combination of
monomials in the set $V$, then $x$ and $y$ commute with the
monomials with nonzero coefficients in that linear combination.
Thus $Z(A)$ has a basis which is a subset of $V$. Clearly
$x$, $y$ commute exactly with the monomials $x^iy^j$ where either
both $i$, $j$ are   divisible  by $e$ or one of $i$, $j$ is $p-1$.
This shows that $Z(A)$ has a basis as stated in (i). Since $x^e$ and $y^e$
are in $J(Z(A))$, hence annihilate $\soc(Z(A))$, it follows that
$\soc(Z(A))$ is contained in the span of the  elements  of  
\[ \{ x^i y^{j} \,   | \, p-e \leq i, j \leq p-1\}  \cap Z(A) =
\{x^{i}y^{p-1}, x^{p-1}y^{j}\ |\ p-e \leq i, j \leq p-1\}. \]    
On the other hand, every element of 
$\{x^{i}y^{p-1}, x^{p-1}y^{j} \ |\ p-e \leq i, j \leq p-1\}$  
is annihilated by $J(Z(A))$, whence (ii).
Statement (iii) follows easily from the previous statements.
\end{proof}

\begin{Lemma} \label{commutatorbasis}
The set
$$\{x^iy^j\ |\ 1\leq i,j\leq p-1,\ i\ \mathrm{or}\ j
\ \mathrm{not  \,  divisible  \,  by \,  }  e \}$$
is a $k$-basis of $[A,A]$. In particular, we have
\[ \dim_k([A,A])=(p-1)^2 -  \textstyle\left(\frac{p-1}{e}\right) ^2 \]
and the space $[A,A]$ is contained in the ideal $Axy=$ $xyA$.
\end{Lemma}

\begin{proof}
Let $1\leq i, j  \leq p-1 $. If $ j$ is not divisible by $e$,  then
$$[x, x^{i-1}y^j] = x x^{i-1}y^j -x^{i-1} y^jx=  
(1- q^{j})  x^iy^j   \ne 0  $$ 
whence  $ x^{i}y^j \in [A,A] $. Similarly, if $i$ is not divisible by 
$e$, then 
\[ x^iy^j = (1-q^i)^{-1}  [y, x^i y^{j-1}] \in [A, A]. \]    
Thus the given set is contained in $[A, A] $  and it spans a subspace 
of $[A,A]$ of dimension $(p-1)^2 - \left(\frac{p-1}{e} \right) ^2 $.   
Since $\dim_k([A,A])=$ $\dim_k(A)-\dim_k(Z(A))$, the formula for 
$\dim_k([A,A])$ follows from Lemma \ref{centralbasis}. This dimension 
coincides with the  dimension of  the subspace spanned  by  the given 
set, whence the result.
\end{proof}

Let $f \colon A \to A$ be a derivation. Then $f(1)=$ $0$, and 
$f$ is uniquely determined by its values at $x$ and $y$. 
An easy induction shows that for any positive integer $n$ and 
$x_1, x_2,\dots,x_n\in A$, we have
$$f(x_1x_2\cdots x_n) = 
\sum_{i=1}^n\ x_1x_2\cdots x_{i-1} f(x_i) x_{i+1}\cdots x_n\ ;$$
in particular, for any $x\in$ $A$ we have
$f(x^n) =$ $\sum_{i=1}^{n}\ x^{i-1} f(x) x^{n-i}$.

\begin{Lemma} \label{derivationLemma}
For $0\leq i,j\leq p-1$, let $\alpha_{i,j}$, $\beta_{i,j}\in$ $k$.
There is a derivation $f \colon A\to A$ satisfying
\[ f(x) = \sum_{0\leq i,j\leq p-1}\ \alpha_{i,j} x^iy^j, \qquad
f(y) = \sum_{0\leq i,j\leq p-1}\ \beta_{i,j} x^iy^j  \]
if and only if the following hold.
\begin{enumerate}
\item[\rm (1)]
$\alpha_{i,j-1}(1-q^{i-1})+\beta_{i-1,j}(1-q^{j-1}) =0$ for $1\leq i,j\leq p-1$.
\item[\rm (2)]
$\alpha_{0,j-1}=0$ for $1\leq j\leq p-1$.
\item[\rm (3)]
$\beta_{i-1,0}=0$ for $1\leq i\leq p-1$.
\end{enumerate}
In particular, if $f$ is a derivation on $A$, then $f$ maps $J(A)$ to $J(A)$.
\end{Lemma}

\begin{proof}
Suppose that $f$ is a derivation with the given values for
$x$ and $y$.
In the following sums,  unless otherwise indicated, the indices $i$ and $j$ 
run from $0$ to $p-1$. We have
\begin{align*} 
0 &=f(0)= f(qxy -yx) = qf(x)y- yf(x)  + qxf(y)-f(y)x \\
 &= \sum_{i,j}\ \alpha_{i,j}(qx^iy^{j+1}-yx^iy^j) +
\sum_{i,j}\ \beta_{i,j}(qx^{i+1}y^j -x^iy^jx) \\
&= \sum_{i,j}\ \alpha_{i,j}(q-q^i)x^iy^{j+1}  + 
\sum_{i,j}\ \beta_{i,j}(q-q^j)x^{i+1}y^j .
\end{align*}

The term in the first sum with $j=$ $p-1$ is zero, as is the
term in the second sum with $i=$ $p-1$ (this is where we use
the $p$-power relations $x^p=$ $0=$ $y^p$). We reindex the first
sum with $j$ running from $1$ to $p-1$, and the second sum with
$i$ running from $1$ to $p-1$, and we then separate the
terms with $i=0$ or $j=0$. This yields
$$ 0= (q-1) \left( \sum_{ j=1} ^{p-1}\alpha_{0, j-1} y^j + 
\sum_{ i=1}^{p-1} \beta_{i-1, 0} x^i \right)   + q \sum_{i,j =1} ^{p-1}\  
( \alpha_{i,j-1}(1-q^{i-1}) +  \beta_{i-1,j}(1-q^{j-1} ) )  x^iy^{j}. $$

Since $q \ne 1 $, the first two sums above yield the conditions (2) and 
(3). The third sum yields the condition (1).
In particular, $\alpha_{0,0}=$
$\beta_{0,0}=$ $0$, and hence $f(x)$, $f(y)\in$ $J(A)$.
This implies that $f$ sends $J(A)$ to $J(A)$. 

Conversely, there is a derivation $g$ from the free algebra 
$k\langle x,y \rangle$ in two generators (abusively again denoted $x$ and
$y$) to the algebra $A$ which takes on $x$ and $y$
the values as given in the statement. By construction, $g$ vanishes
on $qxy-yx$. The properties (1), (2)  and  (3) imply that $g$ 
vanishes also on $x^p$ and $y^p$. Thus $g$ induces a
derivation on $A$ with the required values for $x$ and $y$.
\end{proof}

\begin{Lemma} \label{Derdim}
We have $\dim_k(\Der(A)) = p^2 +1 +  (\frac{p-1}{e} )^2 $.  
\end{Lemma}

\begin{proof}
A derivation $f \colon A\to$ $A$ is determined by the $2p^2$
coefficients $\alpha_{i,j}$, $\beta_{i,j}$ as in
Lemma \ref{derivationLemma}. Any assignment of values $f(x)$, $f(y)$ 
satisfying the conditions (1) to (3) in that lemma
determines a unique derivation.     
If $e$ divides both $i-1 $ and $j-1 $, then the condition (1) is trivially 
satisfied, otherwise (1) yields  a relation. 
Thus the condition (1) yields  $(p-1)^2- (\frac{p-1}{e})^2  $
relations. The conditions (2) and (3) each yield $p-1 $
relations. Thus, the total number of relations from 
Lemma~\ref{derivationLemma} is 
\[ (p-1)^2- \textstyle\left(\frac{p-1}{e}\right)^2  + 2(p-1) = p^2 -1
-\left(\frac{p-1}{e}\right)^2 \]  
and it follows that 
$\dim_k(\Der(A))=$ $p^2 +1 +  (\frac{p-1}{e})^2  $. 
\end{proof}

\begin{proof}[Proof of Proposition \ref{HHonedim}]
We have $\dim_k(A)=$ $p^2$ and 
$\dim_k(Z(A))=  \left(\frac{p-1}{e} \right)^2+2p-1$
from Lemma~\ref{centralbasis}. Thus $\dim_k(\IDer(A))=$
$p^2- \left(\frac{p-1}{e} \right)^2 -2p + 1$. It follows 
from Lemma \ref{Derdim} that 
\[ \dim_k(\HH^1(A))=  2\textstyle\left(\frac{p-1}{e}\right)^2 +2p, \] 
which completes the proof of Proposition \ref{HHonedim}.
\end{proof}

\section{The Lie algebra structure of $\HH^1(A)$}\label{LieStructure}

A Lie subalgebra $\CH$ of a Lie algebra $\CL$ is called {\it toral}
if the image of $\CH$ in the adjoint representation on $\CL$
is simultaneously diagonalisable (hence abelian). For semisimple
complex Lie algebras, the maximal toral Lie subalgebras are 
exactly the Cartan subalgebras. As in the previous section, let $k$ be 
a field of odd prime characteristic $p$,  let $1\ne q \in k^{\times}$ 
have order $e$ dividing $p-1$, and let
$$A = k \langle x,y\ |\ x^p=y^p=0,\ yx=qxy\rangle.$$

The technicalities needed for the proof of Theorem
\ref{HHoneLiestructure} are contained in the 
following series of lemmas. We start by identifying
innner derivations.

\begin{Lemma} \label{innerder}
For $i$, $j$ such that $0\leq i,j\leq p-1$, consider the
inner derivation $d_{i,j} = [x^iy^j, - ]$ on $A$.
\begin{enumerate}
\item[\rm (i)]
We have $d_{i,j}(x)=$ $(q^j-1)x^{i+1}y^j$, 
where $0\leq$ $i,j\leq$ $p-1$. In particular, we have
$d_{i,j}(x)=0$ if and only if $i=p-1$ or $e$ divides $j$.
\item[\rm (ii)] 
We have $d_{i,j}(y)=$ $(1-q^i)x^{i}y^{j+1}$,
where $0\leq$ $i,j\leq$ $p-1$. In particular, we have
$d_{i,j}(y)=0$ if and only if $j=p-1$ or $e$ divides $i$.
\item[\rm (iii)] 
Let $d$ be an inner derivation of $A$. 
Then $d(x)$ is a linear combination of monomials $x^iy^j$ with 
$1\leq$ $i,j\leq$ $p-1$ such that $e$ does not divide $j$. Similarly,
$d(y)$ is a linear combination of monomials $x^iy^j$ with 
$1\leq$ $i,j\leq$ $p-1$ such that $e$ does not divide $i$. 
\end{enumerate}
\end{Lemma}

\begin{proof}
Let $i$, $j$ be integers such that $0\leq i,j\leq p-1$. We have
\[ d_{i,j}(x)= [x^iy^j,x]= x^iy^jx-x^{i+1}y^j=(q^j-1)x^{i+1}y^j. \] 
This expression vanishes precisely if $q^j=1$
or if $i+1=p$, whence (i). As similar calculation proves (ii).
An inner derivation on $A$ is a linear combination of the inner
derivations $d_{i,j}$, where $0\leq$ $i,j\leq$ $p-1$. Thus
(iii) follows from (i) and (ii).
\end{proof}

Using Lemma \ref{derivationLemma}, we determine all derivations on $A$
mapping one of the generators to a single monomial and the other
to zero.

\begin{Lemma}\label{monomialderivation}
Let $a$, $b$, $c$, $d$ be integers such that 
$0\leq a,b,c,d\leq p-1$.
\begin{enumerate}
\item[\rm (i)]
There is a derivation $f_{a,b}$ on $A$ satisfying $f_{a,b}(x)=$
$x^ay^b$ and $f_{a,b}(y)=0$ if and only if $b=p-1$ or $a\geq$ $1$
and $e$ divides $a-1$. Moreover, in that case we have
$$f_{a,b}(x^cy^d) = \left(\sum_{s=0}^{c-1} q^{bs}\right) x^{a+c-1}y^{b+d}\ ,$$
with the convention that this is zero if $c=0$.
In particular, if $e$ divides $a-1$ and $b$ or if
$b=p-1$, then
$$f_{a,b}(x^cy^d) =  c x^{a+c-1}y^{b+d}\ .$$
\item[\rm (ii)]
There is a derivation $g_{a,b}$ on $A$ satisfying $g_{a,b}(x)=$
$0$ and $g_{a,b}(y)=x^ay^b$ if and only if $a=p-1$ or $b\geq$ $1$
and $e$ divides $b-1$. Moreover, in that case we have
$$g_{a,b}(x^cy^d) = \left(\sum_{t=0}^{d-1} q^{at}\right) x^{a+c}y^{b+d-1}\ ,$$
with the convention that this is zero if $d=0$.
In particular, if $e$ divides $a$ and $b-1$, or if $a=p-1$, then
$$g_{a,b}(x^cy^d) = d x^{a+c}y^{b+d-1}\ .$$
\end{enumerate}
\end{Lemma}

\begin{proof}
With the notation of Lemma \ref{derivationLemma}, the condition
$f_{a,b}(y)=0$ is equivalent to the vanishing of all coefficients
$\beta_{i,j}$, where $0\leq$ $i,j\leq$ $p-1$. The condition
$f_{a,b}(x)=$ $x^ay^b$ is equivalent to $\alpha_{a,b}=1$ and the
vanishing of all remaining coefficients $\alpha_{i,j}$. 
If $1\leq$ $a\leq$ $p-1$ and $0\leq$ $b\leq$ $p-2$, then
the relation (1) from Lemma \ref{derivationLemma} yields
$0=$ $\alpha_{a,b}(1-q^{a-1})=$ $1-q^{a-1}$, hence that $e$
divides $a-1$. If $a=0$, then relation (2) from Lemma \ref{derivationLemma}
forces $b=$ $p-1$. Suppose now that $f_{a,b}$ is a derivation;
that is, $b=p-1$ or $a\geq$ $1$ and $e$ divides $a-1$. Since
$f_{a,b}(y)=0=f_{a,b}(1)$, an easy induction shows that 
$f_{a,b}(y^d)=0$.
Thus $f_{a,b}(x^cy^d)=$ $f_{a,b}(x^c)y^d$. Another straighforward
induction shows that $f_{a,b}(x^c)=$ 
$(\sum_{s=0}^{c-1} q^{bs}) x^{a+c-1}y^{b}$. Combining these 
facts yields the first formula in (i). If in addition $e$ 
divides $b$, then $q^b=1$, whence the second formula.
This proves (i), and the proof of (ii) is similar.
\end{proof}

Note the slight redundancy in the statement of 
Lemma \ref{monomialderivation}: if $0\leq$ $a\leq$ $p-1$ and $e$
divides $a-1$, then necessarily $a\geq$ $1$, since we assume that
$e\neq$ $1$. We determine next a linearly independent subset of
$\Der(A)$ whose image in $\HH^1(A)$ is a $k$-basis.

\begin{Lemma} \label{HHonebasis}
Let $a$, $b$, $c$, $d$ be integers such that $0\leq a,b,c,d\leq p-1$. 
Let $X$ be the disjoint union of the two sets of derivations 
\begin{gather*}
\{f_{a,b}\ |\ 0\leq a, b\leq p-1,\ 
e\ \mathrm{divides}\ a-1\ \mathrm{and}\ b,\ \mathrm{or}
\ b=p-1\} \\
\{g_{a,b}\ |\ 0\leq a, b\leq p-1,\ 
e\ \mathrm{divides}\ a\ \mathrm{and}\ b-1, \ \mathrm{or} 
\ a=p-1\}
\end{gather*}
The set $X$ is linearly independent, and its span 
$\mathcal{H}$ is a complement of $\IDer(A)$ in $\Der(A)$.
\end{Lemma}

\begin{proof}
By Lemma \ref{monomialderivation}, the set $X$ indeed consists of
derivations.
The linear independence of the set $X$ of derivations  follows 
immediately from the fact that the set $V$ of monomials 
in $x$ and $y$ is a basis of $A$. The cardinality of the 
set $X$ is equal to $\dim_k(\HH^1(A))$, by Proposition \ref{HHonedim}.
Any nonzero linear combination of the derivations in $X$ 
map either $x$ or $y$ to a nonzero element in $A$. If $x$ 
is mapped to a nonzero element, this element involves a 
monomial $x^ay^b$ with $b$ divisible by $e$. But then Lemma \ref{innerder}
implies that this linear combination is not an inner derivation.
A similar argument applies if $y$ is mapped to a nonzero element. 
This shows that the space $\mathcal{H}$ spanned by $X$ intersects 
$\IDer(A)$ trivially. Since $\dim_k(\mathcal{H})=$ $|X|=$ 
$\dim_k(\HH^1(A))$, it follows that $\mathcal{H}$ is a complement 
of $\IDer(A)$ in $\Der(A)$. 
\end{proof}

We calculate next the Lie brackets between the elements of the basis
$X$ of $\mathcal{H}$. 

\begin{Lemma} \label{derLiestructure}
With the notation of Proposition \ref{HHonebasis}, let $a$, $b$, $c$, 
$d$ be integers such that $0\leq a,b,c,d\leq p-1$.
\begin{enumerate}
\item[\rm (i)]
Suppose that $e$ divides $a-1$ and $b$, or that $b=p-1$;
similarly, suppose that $e$ divides $c-1$ and $d$, 
or that $d=p-1$. If $a+c-1\leq$ $p-1$ and $b+d\leq$ $p-1$, then
$$[f_{a,b},f_{c,d}]= (c-a)f_{a+c-1,b+d}$$ 
and we have $b+d=p-1$ or $e$ divides both $a+c-2$ and $b+d$;
in particular, we have $f_{a+c-1,b+d}\in$ $X$. If
one of $a+c-1$ or $b+d$ is at least $p$, then
$$[f_{a,b},f_{c,d}]= 0\ .$$
\item[\rm (ii)]
Suppose that $e$ divides $a$ and $b-1$, or that $a=p-1$;
similarly, suppose that $e$ divides $c$ and $d-1$, or that
$c=p-1$. If $a+c\leq p-1$ and $b+d-1\leq p-1$, then
$$[g_{a,b},g_{c,d}]= (d-b)g_{a+c,b+d-1}$$ 
and we have $a+c=p-1$ or $e$ divides both $a+c$ and $b+d-2$;
in particular, we have $g_{a+c,b+d-1}\in$ $X$.
If one of $a+c$, $b+d-1$ is at least $p$, then 
$$[g_{a,b},g_{c,d}]= 0\ .$$
\item[\rm (iii)]
Suppose that $e$ divides $a-1$ and $b$, or that $b=p-1$;
similarly, suppose that $e$ divides $c$ and $d-1$, or that
$c=$ $p-1$. 
If $a+c>p-1$ or $b+d>p-1$, then
$$[f_{a,b}, g_{c,d}]= 0\ .$$
\item[\rm (iv)]
Suppose that $e$ divides $a-1$ and $b$, or that $b=p-1$;
similarly, suppose that $e$ divides $c$ and $d-1$, or that
$c=$ $p-1$. Suppose that $a+c\leq p-1$ and that $b+d\leq p-1$.
We have $a+c<p-1$ if and only if $b+d<p-1$, and in that case, 
we have
$$[f_{a,b}, g_{c,d}]= -bf_{a+c,b+d-1}+cg_{a+c-1,b+d}\ .$$ 
\item[\rm (v)]
Suppose that $e$ divides $a-1$ and $b$, or that $b=p-1$;
similarly, suppose that $e$ divides $c$ and $d-1$, or that
$c=$ $p-1$. Suppose that $a+c\leq p-1$ and that $b+d\leq p-1$.
We have $a+c=p-1$ if and only if $b+d=p-1$, and
in that case we have $(a,b,c,d)=$ $(0,p-1,p-1,0)$, and
$$[f_{0,p-1},g_{p-1,0}]= (q^{-1}-1)^{-1} [x^{p-2}y^{p-2},-]=
(q^{-1}-1)^{-1}d_{p-2,p-2}\ .$$
\end{enumerate}
\end{Lemma}

\begin{proof}
With the assumptions as in (i), both sides vanish at $y$, and
we need to show that they coincide at $x$. It follows from 
Lemma \ref{monomialderivation} (i) that 
$$[f_{a,b},f_{c,d}](x)= f_{a,b}(x^cy^d)-f_{c,d}(x^ay^b)=
cx^{a+c-1}y^{b+d}-ax^{a+c-1}y^{b+d}$$ 
This is a nonzero derivation only if $a+c-1\leq p-1$ and $b+d\leq$ 
$p-1$. If $b+d<p-1$,
then $b<p-1$ and $d<p-1$, hence $a-1$, $c-1$, $b$, $d$ are divisible
by $e$, and therefore $a+c-2$ and $b+d$ are divisible by $e$. This 
shows (i), and the proof of (ii) is similar.
With the assumptions as in (iii), we have
$$[f_{a,b},g_{c,d}](x) = f_{a,b}(g_{c,d}(x))-g_{c,d}(f_{a,b}(x))= 
-g_{c,d}(x^ay^b)=-bx^{a+c}y^{b+d-1}$$ 
where the last equation uses Lemma \ref{monomialderivation} (ii).
A similar calculation 
yields $[f_{a,b},g_{c,d}](y)=$ $cx^{a+c-1}y^{b+d}$, whence 
(iii). If $a+c=$ $p-1$, then $e$ divides $a$ (since $e$
divides $c$ and $p-1$), so $e$ does not divide $a-1$, and hence 
$b=$ $p-1$. The hypothesis $b+d\leq$ $p-1$ forces $d=$ $0$, so
$e$ does not divide $d-1$, and hence hence $c=$ 
$p-1$, which in turn forces $a=$ $0$ by the hypothesis $a+c\leq$ $p-1$. 
This shows that under the assumptions in (iv) and (v), we have 
$a+c=p-1$ if and only if $b+d=p-1$, which in turn holds if and only if 
$(a,b,c,d)=$ $(0,p-1,p-1,0)$. 
For the proof of (iv), assume that $a+c<p-1$ and $b+d<p-1$. Then 
all of $a$, $b$, $c$, $d$ are strictly smaller than $p-1$.
Thus $e$ divides $a-1$, $d-1$, $b$, and $c$, and therefore $e$ divides
$b+d-1$ and $a+c-1$. Hence $f_{a+c,b+d-1}$ and $g_{a+c-1,b+d}$ are in 
$X$. We have 
$$[f_{a,b}, g_{c,d}](x) = f_{a,b}(0)-g_{c,d}(x^ay^b) =
-bx^{a+c}y^{b+d-1}\ ,$$
where the last equation is from Lemma \ref{monomialderivation} (ii).
This is equal to ${}-bf_{a+c,b+d-1}+cg_{a+c-1,b+d}$ evaluated at $x$.
Similarly,  
$$[f_{a,b}, g_{c,d}](y) = f_{a,b}(x^cy^d)-g_{c,d}(0) =
cx^{a+c-1}y^{b+d}\ ,$$
where the last equation is from Lemma \ref{monomialderivation} (i).
This is equal to ${}-bf_{a+c,b+d-1}+cg_{a+c-1,b+d}$ evaluated at $y$.
The formula in (iv) follows. In order to prove (v), we need to
calculate
$$[f_{0,p-1},g_{p-1,0}](x) = f_{0,p-1}(0)-g_{p-1,0}(y^{p-1})=
x^{p-1}y^{p-2}\ ,$$
where the last equation uses Lemma \ref{monomialderivation} (ii) and
$-(p-1)=1$ in $k$. Similarly, we have
$$[f_{0,p-1},g_{p-1,0}](y) = f_{0,p-1}(x^{p-1})-g_{p-1,0}(0)=
-x^{p-2}y^{p-1}\ .$$
Note that $q^{p-2}=$ $q^{-1}$ since $e$ divides $p-1$.
By Lemma \ref{innerder}, we have
\[ d_{p-2,p-2}(x)=(q^{-1}-1)x^{p-1}y^{p-2}, \qquad
d_{p-2,p-2}(y)=-(q^{-1}-1)x^{p-2}y^{p-1}. \]
Statement (v) follows.
\end{proof}

The space $\mathcal{H}$ in Lemma \ref{HHonebasis} is not a Lie subalgebra of 
$\Der(A)$ because of the relation (v) in \ref{derLiestructure};
this relation implies that the images of $f_{0,p-1}$ and $g_{p-1,0}$
in $\HH^1(A)$ commute (because their Lie bracket is an inner
derivation). The Lie brackets between basis elements in $X$
determine the Lie algebra structure of $\HH^1(A)$. In order to
describe this structure, we first identify those elements in $X$
which are commutators.

\begin{Lemma} \label{HHonederivedcenter}
Let $a$, $b$, $c$, $d$ be integers such that 
$0\leq a,b,c,d \leq p-1$.
\begin{enumerate}
\item[\rm (i)]
If $e$ divides $a-1$ and $b$, or if $b=p-1$, then
$$[f_{1,0},f_{a,b}]= (a-1)f_{a,b}\ 
\text{and}\ [g_{0,1},f_{a,b}]= bf_{a,b}\ .$$
In particular, if $e$ divides $b$, then 
$$[f_{1,0},f_{1,b}]= 0\ 
\text{and}\ [g_{0,1},f_{1,b}]= bf_{1,b}\ .$$
\item[\rm (ii)]
Suppose that $e$ divides $c$ and $d-1$, or that $c=p-1$. We have
$$[f_{1,0},g_{c,d}]=cg_{c,d}\ 
\text{and}\ [g_{0,1},g_{c,d}]= (d-1)g_{c,d}\ .$$
In particular, if $e$ divides $c$, then 
$$[f_{1,0},g_{c,1}]=cg_{c,1}\ 
\text{and}\ [g_{0,1},g_{c,1}]= 0\ .$$
\item[\rm (iii)]
The linear endomorphisms $\ad(f_{1,0})$ and $\ad(g_{0,1})$ of 
$\Der(A)$ restrict to linear endomorphisms of the subspace $\CH$ 
spanned by $X$, and, with respect to the basis $X$, these endomorphisms 
of $\CH$ are represented by diagonal matrices.
\item[\rm (iv)]
All basis elements in $X$ except $f_{1,0}$ and $g_{0,1}$ 
are commutators in $\Der(A)$.
\item[\rm (v)]
We have $[f_{1,0},g_{0,1}]=0$.
\end{enumerate}
\end{Lemma}

\begin{proof}
The statements (i) and (ii) are special cases of Lemma \ref{derLiestructure},
and the statements (iii), (iv), and (v)  follow from (i) and (ii).
\end{proof}

\begin{Lemma} \label{Lprimebasis}
Let $a$, $b$, $c$, $d$ be integers such that $0\leq a,b,c,d \leq p-1$.
\begin{enumerate}
\item[\rm (i)]
Suppose that $e$ divides $a-1$ and $b$, or that $b=p-1$.
If $a+b\geq$ $2$, then $a\geq$ $e+1$ or $b\geq$ $e$.
In particular, $a+b-1\geq\min\{e,p-2\}$.
\item[\rm (ii)]
Suppose that $e$ divides $c$ and $d-1$, or that $c=p-1$.
If $c+d\geq$ $2$, then $c\geq$ $e$ or $d\geq$ $e+1$.
In particular, $c+d-1\geq\min\{e,p-2\}$.
\end{enumerate}
\end{Lemma}

\begin{proof} 
Assume that $b<e$. Then $b<p-1$, so $e$ divides $a-1$ and $b$. 
The inequality $b<e$ forces $b=0$. Since $a+b\geq$ $2$, this
implies $a\geq$ $2$, hence $a-1\geq$ $1$. Since $e$ divides
$a-1$, it follows that $a-1\geq$ $e$, and hence $a+b-1\geq$ $e$.
If $b\geq$ $p-1$, then $a+b-1\geq$ $p-2$, whence (i). A similar
argument yields (ii).
\end{proof} 

\begin{proof}[Proof of Theorem \ref{HHoneLiestructure}]
Statement (i) is proved in Proposition \ref{HHonedim}.
We use the same notation as in Theorem \ref{HHoneLiestructure}; in particular,
$\CL=$ $\HH^1(A)$ and $\CL'$ is the derived Lie subalgebra of $\CL$.
It follows from Lemma~\ref{HHonederivedcenter} that $\CL'$ contains
the images of all elements of $X$ except possibly
the images of $f_{1,0}$ and $g_{0,1}$. 

The relations in Lemma \ref{derLiestructure} imply that $\CL'$ contains no
nonzero linear combination of the images of $f_{1,0}$ and $g_{0,1}$. 
Thus $\CL'$ has codimension $2$ in $\CL$.

A complement of $\CL'$ is spanned by the image of 
$\{f_{1,0}, g_{0,1}\}$, and this complement is a $2$-dimensional abelian 
Lie subalgebra of $\CL$, by Lemma \ref{HHonederivedcenter} (v). Moreover,
$\CL'$ has as a basis the image in $\HH^1(A)$ of the set 
$$X'=X\setminus \{f_{1,0},g_{0,1}\} \ .$$
Equivalently, $X'$ consists of all $f_{a,b}$, $g_{c,d}$ in 
$X$ with $a+b\geq$ $2$ and $c+d\geq$ $2$. 

It follows from Lemma \ref{HHonederivedcenter} (iii) that the images of 
$f_{1,0}$ and $g_{0,1}$ span a toral subalgebra. 
Lemma~\ref{HHonederivedcenter} implies that the 
centraliser in $\CL$ of the image of $f_{1,0}$ is spanned by the images 
of $f_{1,b}$, $g_{0,d}$ with $b$ and $d-1$ divisible by $e$. Similarly, 
the centraliser in $\CL$ of the image of $g_{0,1}$ is spanned by the 
images of $f_{a,0}$, $g_{c,1}$, with $a-1$ and $c$ divisible
by $e$. Thus $Z(\CL)$ is contained in the span of the images of 
$f_{1,0}$ and $g_{0,1}$, but it follows again from 
Lemma \ref{HHonederivedcenter} that no nonzero linear combination of these 
two elements is in the center. This shows that $Z(\CL)=$ $\{0\}$ and 
that the toral subalgebra $\CH$ is maximal. This proves (ii) and (iii).

For $m\geq$ $1$ denote by $\CL_m$ the subspace of $\CL$ spanned by the 
images of those $f_{a,b}$, $g_{c,d}$ in $X$ for which $a+b\geq m$ and 
$c+d \geq m$. Thus $\CL_1=$ $\CL$,
$\CL_2=$ $\CL'$, and $\CL_m=$ $\{0\}$ for $m\geq$ $2p$. 
The relations in Lemma \ref{derLiestructure}
imply that $[\CL',\CL_m]\subseteq$ $\CL_{m+1}$, which in turn
implies that $\CL'$ is nilpotent, whence (iv).

The socle of $\CL$ as a $Z(A)$-module is contained in the subspace of 
$\CL$ which is annihilated by $x^e$ and $y^e$. 
We have $x^ef_{a,b}=$ $f_{a+e,b}$ if $a+e\leq p-1$, and
$x^ef_{a,b}=0$ if $a\geq$ $p-e$. Similarly, we have 
$y^ef_{a,b}=$ $f_{a,b+e}$ if $b+e\leq$ $p-1$ and 
$y^ef_{a,b}=0$ if $b\geq$ $p-e$.
It follows that the socle of $\CL$ as a $Z(A)$-module
is equal to the subspace of $\HH^1(A)$ which is annihilated by $x^e$ and 
$y^e$. 
Thus the image of $f_{a,b}$ in $\CL$ is contained in 
$\soc_{Z(A)}(\CL)$ if $a\geq$ $p-e$ and $b\geq$ $p-e$.
Since also $e$ divides both $a-1$ and $b$ or $b=p-1$,
this forces $b=p-1$. Similarly, the image of $g_{a,b}$ in $\CL$
is contained in $\soc_{Z(A)}(\CL)$ if and only if $b\geq$ $p$
and $a=p-1$.  It follows that $\soc_{Z(A)}(\CL)$ is equal to the 
space spanned by the image in $\CL$ of the set
$$S = \{f_{a,p-1}\ |\ p-e\leq a\leq p-1\}\cup
\{g_{p-1,b}\ |\ p-e\leq b\leq p-1\}$$
This shows in particular that
$$\dim_k(\soc_{Z(A)}(\CL)) = 2e\ .$$ 
The relations in Lemma \ref{derLiestructure} imply that we have
$[X',S]=$ $\{0\}$, and hence we have an inclusion
$$\soc_{Z(A)}(\CL) \subseteq Z(\CL')\ .$$ 
This proves (v).

By the above and Lemma \ref{Lprimebasis},  $\CL'$ is spanned by the 
images of elements $f_{a,b}$, $g_{c,d}$ where at least one of $a$, $b$ 
is greater or equal to $e$, and where at least one of $c$, $d$ is 
greater or equal to $e$. Thus $\CL'$ is contained in $x^e\CL+y^e\CL$. 
Since no nonzero linear combination of the images of
$f_{1,0}$, $g_{0,1}$ is contained in $J(Z(A))\CL$, statement
(vi) follows.

In order to prove (vii), suppose first that $e=p-1$. In that case we have
$$X'=\{f_{a,p-1}\ |\ 0\leq a\leq p-1\}\cup 
\{g_{p-1,b}\ |\ 0\leq b\leq p-1\}$$ 
The images in $\CL$ of any two elements of $X'$ commute; more
precisely, any two elements in $X'$ commute already in $\CH$, except 
for $[f_{0,p-1},g_{p-1,0}]$, which is inner by Lemma \ref{derLiestructure} 
(v). This shows that $\CL'$ is abelian if $e=p-1$. 

Suppose that $e<p-1$; in particular, $e\leq$ $\frac{p-1}{2}$.
We consider the basis $X'=$ $X\setminus \{f_{1,0}, g_{0,1}\}$
of $\CL'$. Since $e<p-1$, there are derivations $f_{e+1,0}$,
$g_{0,e+1}$, $f_{1,e}$, and $g_{e,1}$ in $X'$. 
Using Lemma~\ref{derLiestructure}, one  verifies that the Lie brackets of 
any of these four elements with any element in $X'$ yield elements in 
$X$, possibly multiplied by scalars (which can be zero). It follows 
that in order to calculate centralisers in $\CL'$ of these four
particular elements, it suffices to 
calculate centralisers in the space $\CH'$ spanned by $X'$. 
It follows further that if one of the above four elements
centralises a linear combination of elements in $X'$, it
centralises the elements of $X'$ with nonzero coefficients
individually.   
A tedious verification, using Lemma \ref{derLiestructure}, shows
that the centraliser of $f_{e+1,0}$ in $X'$ intersected with
the centraliser of $g_{0,e+1}$ is the set $S_1\cup S_2$, where
$$S_1 = \{f_{a,p-1}\ |\ p-e\leq a\leq p-1\}\cup 
\{f_{e+1,0}, f_{p-e,0}, f_{e+1,p-1}\}$$
and 
$$S_2= \{g_{p-1,d}\ |\ p-e\leq d\leq p-1\}\cup 
\{g_{0,e+1}, g_{0,p-e}, g_{p-1,e+1}\}$$
The element $f_{1,e}$ does not centralise any of the 
two elements $f_{e+1,0}$ and $f_{p-e,0}$. Similarly, $g_{e,1}$
centralises neither $g_{0,e+1}$ nor $g_{0,p-e}$. Thus
every element in $Z(\CL')$ is the image of a linear combination 
of the set
$$S_3 = \{f_{a,p-1}\ |\ p-e\leq a\leq p-1\}\cup 
\{f_{e+1,p-1}\} \cup
\{g_{p-1,d}\ |\ p-e\leq d\leq p-1\}\cup 
\{g_{p-1,e+1}\}$$
One verifies that the image of $S_3$ is contained in $Z(\CL')$.
The cardinality of $S_3$ is $2e+2$. Statement (vii) follows.

For the last statement, note that the $p$-power map on $\CL$ is
induced by the map sending a derivation $f$ on $A$ to the composition 
$f^{[p]}=$ $f\circ f\circ\cdots \circ f$ of $f$ with itself $p$ times. 
We clearly have $(f_{1,0})^{[p]}=$
$f_{1,0}$, and $(g_{0,1})^{[p]}=$ $g_{0,1}$; that is, the images of
$f_{1,0}$ and $g_{0,1}$ in $\CL$ are $p$-toral. Since the image of 
$\{f_{1,0},g_{0,1}\}$ in $\CL$ is a basis of $\CH$, this shows that
$\CH$ is $p$-toral. Any element of $X'=$ $X\setminus\{f_{1,0}, g_{0,1}\}$
is of the form $f_{a,b}$ or $g_{a,b}$ with $a+b-1\geq$ $e$ or
$a+b-1\geq p-2$ (the latter arises if $a$ or $b$ is equal to $p-1$). 
Consider first the case where $p\geq$ $5$, so that $p-2\geq$ $3$. Since 
$e\geq$ $2$, it follows that $a+b-1\geq$ $2$. Lemma 
\ref{monomialderivation} implies that any derivation in $X'$ sends a 
monomial in $x$, $y$ of total degree $m$ to a scalar multiple of a 
monomial of total degree at least $m+2$. Thus any composition of $p$ 
elements in $X'$ sends a monomial in $x$, $y$ of degree at least $1$ to 
a scalar multiple of a monomial $x^cy^d$ of total degree $c+d\geq$ 
$1+2p$. This implies that at least one of $c$, $d$ is greater that $p$, 
which in turn implies that $x^cy^d=0$ in $A$. It follows that any 
composition of $p$ elements in $X'$ is zero. Therefore, if $f$ is a 
linear combination of elements in $X'$, then $f^{[p]}=0$. Since the 
image in $\CL$ of $X'$ is a basis of $\CL'$, this proves (viii) in the 
case $p\geq$ $5$. If $p=3$, then $e=2$, and we have 
$$X'= \{f_{0,2}, f_{1,2},f_{2,2}, g_{2,0}, g_{2,1},g_{2,2}\}\ .$$
A direct verification shows that the composition of any three 
derivations in this set is zero, completing the proof.  
\end{proof}

\begin{proof}[{Proof of Corollary \ref{radicaldimCor}}]
A stable equivalence of Morita type preserves the
Tate analogue of Hochschild cohomology, hence
preserves $\HH^1(A)$ as a module over $\HH^0(A)\cong$
$Z(A)$ since the projective ideal in $Z(A)$ annihilates
Hochschild cohomology in positive degrees. The
corollary follows from statement (v) in Theorem
\ref{HHoneLiestructure} together with Corollary
\ref{soclederivationlocal}.
\end{proof}

\begin{proof}[{Proof of Corollary \ref{blockradicaldimCor}}]   
First consider the case that $B$ is nilpotent. By the structure
theorem of nilpotent blocks \cite[(1.4.1)]{Puig}, $B$ is a matrix algebra over $kP$.  
Since $P$ is either cyclic or elementary abelian of order $p^2$,  
$\dim_k(J(B)/J(B)^2)$ equals $1$ or $2$, and the result holds.    
Thus,  we may assume that  $B$ is not nilpotent. In particular,  
since $P$ is abelian, $I > 1$ \cite[(1.ex.3)]{BrPu}. 
If $P$ is cyclic, then since    
$C$ has a unique isomorphism class of simple modules, $ I=1 $, 
a contradiction. Thus, we may assume that $P$ is elementary
abelian. By \cite[Theorem~1.1]{KeLi}, the inertial quotient    
of $B$ is abelian. 
By the structure theory of blocks with normal defect group 
(\cite[Theorem A]{Ku} or \cite[\S45]{Thev}), $C$ 
is a matrix algebra over a twisted group algebra of the semidirect
product of $P$ with the inertial quotient of $B$.
Hence, since $C$ has a unique isomorphism class of simple modules, 
by \cite[Lemma 2]{DeJa}, the inertial 
quotient of $B$ is a direct product of two cyclic groups of order 
$\sqrt I$ (see for instance the  proof of Theorem 1.1 and  
Proposition 5.3 of \cite{KeLi}).  By Theorem 4.2 and Corollary 4.3  
of \cite{HoKe} and their proofs, a basic algebra of $C$ is    
isomorphic to  the algebra $A$ of Theorem \ref{HHoneLiestructure}  
with $1<e \le \sqrt I $.   By \cite[Theorem 1.1]{KeLi},   $B$ is local. Finally, by \cite[Theorem A.2] {Li09},  
there is a stable equivalence of Morita type between $B$ and $C$.  
The result now follows from Corollary \ref{radicaldimCor}.
\end{proof}

\section{Lifting quantum complete intersections over $\CO$}
\label{liftingSection}

Let $p$ be an odd prime and $\CO$ a complete discrete valuation
ring containing a primitive $4p$-th root of unity, with residue field 
$k$ of characteristic $p$ and field of fractions $K$ of characteristic 
$0$.

We denote in this section by $G$ a finite group obtained as
a semi-direct product of an elementary abelian $p$-group
$$P= \langle g\rangle \times \langle h\rangle \cong C_p\times C_p$$
of rank two by a quaternion
group $Q_8=$  $\langle s,t\ |\ s^4=1, s^2=t^2, sts^3=t^3\rangle$
of order $8$, acting on $P$ by $sgs^{-1}=g^{-1}$, $shs^{-1}=h$, 
$tgt^{-1}=g$, and $tht^{-1}= h^{-1}$. In particular, the unique
central involution $z=$ $s^2=$ $t^2$ of $Q_8$ acts trivially on $P$,
hence $Z(G)= \langle z\rangle$. The group algebra $\OG$ has
two blocks, the principal block $B_0=\OG e_0$, where
$e_0=\frac{1}{2}(1+z)$, and one nonprincipal block 
$B_1=\OG e_1$, where $e_1=\frac{1}{2}(1-z)$. The
block $B_1$ has a unique isomorphism class of simple modules,
and more precisely, the quantum complete intersection 
$$A = k\langle x, y\ |\ x^p=y^p=0,\ xy+yx=0\rangle$$ 
is a basic algebra of $k\tenO B_1$. We determine the structure
of a basic algebra of $B_1$. To do this, we will require the
Chebyshev polynomials $T_n$ of the first kind.
For $n\geq$ $0$, the polynomial $T_n$ in the variable $u$ is the unique
polynomial in $\Z[u]$ of degree $n$ satisfying $T_n(\cos(\theta))=$
$\cos(n\theta)$ for any $\theta\in$ $\R$. Using $\sin(\theta)=$
$\cos(\theta-\frac{\pi}{2})$ we obtain for $n$ odd the formula
$$\sin(n\theta) = (-1)^{\frac{n-1}{2}}T_n(\sin(\theta))\ .$$
The polynomials $T_n$ can be defined recursively by $T_0(u)=$ $1$, 
$T_1(u)=$ $u$, and $T_{n+1}(u) = uT_n(u)-T_{n-1}(u)$ for $n\geq$ $1$.
This recursion formula shows that the leading coefficient of
$T_n$ is $2^{n-1}$. It also shows that for $n$ even (resp. odd), the 
polynomial $T_n$ is involves only even (resp. odd) powers of the
variable $u$. 
For $n\geq$ $0$, define a polynomial $f_n$ in the variable $u$ by
$$f_n(u) = 2T_n(\textstyle\frac{u}{2})\ .$$
Then $f_0(u)=2$, $f_1(u)=$ $u$, and $f_{n+1}(u)=$ $uf_n(u)-f_{n-1}(u)$.
In particular, $f_n$ is a polynomial in $\Z[u]$ with leading coefficient 
$1$, and if $n$ is even (resp. odd), then $f_n$ involves only even (resp. 
odd) powers of $u$. 
The well-known explicit formulae for Chebyshev polynomials imply 
that if $n=p$, then all coefficients of $f_p$ other than the leading
coefficient of $f_p$ are divisible by $p$, and hence $f_p$ reduces
to the monomial $u^p$ in $k[u]$. 

\begin{Theorem} 
With the notation above, let $\hat A$ be the $\CO$-algebra
$$\hat A = \CO\langle \gamma, \delta \ |\ \gamma\delta+\delta\gamma=0,
\ f_p(\gamma)=0=f_p(\delta)\rangle$$  
Then $\hat A$ is a basic algebra of $B_1$; in particular, we have
$k\tenO \hat A\cong$ $A$.
\end{Theorem}

\begin{proof}
Since $e_1$ annihilates the $\CO Q_8$-modules of rank one, it follows 
that $S=$ $\CO Q_8e_1$ is the quotient algebra of $\CO Q_8$ 
corresponding to the unique irreducible character of $Q_8$ of degree 
$2$, hence isomorphic to the matrix algebra $M_2(\CO)$. 
The unique simple $B_1$-module (up to isomorphism) has dimension $2$. 
Thus, setting $\hat A=$ $C_{B_1}(S)=$ $B_1^{Q_8}$, we get from
\cite[(7.5) Proposition]{Thev} that
$$B_1 = S \tenO \hat A$$
and then necessarily $\hat A$ is a basic algebra of $B_1$, as its 
unique simple module has dimension $1$. We need to show that 
$\hat A$ has generators satisfying the relations as in the statement, 
and then we need to show that there are no other relations. We use the 
generators $g$, $h$, $s$, $t$ of the group $G$ with the relations as 
stated at the beginning of this section.  Define elements $\gamma$ and 
$\delta$ in $B_1$ by
$$\gamma = (g-g^{-1})te_1\ ,\ \ \ \ \delta = (h-h^{-1})se_1\ .$$ 
Note that $t$ commutes with $g$, $g^{-1}$, hence with $g-g^{-1}$. 
Similarly, $s$ commutes with $h$, $h^{-1}$, and with $h-h^{-1}$. 
We have
\begin{align*}
ste_1&=\half st(1-t^2)=\half s(t-t^{-1})\\
tse_1&=\half st^3(1-t^2)=\half s(t^{-1}-t)
\end{align*}
and hence $tse_1=-ste_1$. Using this equality, we 
verify that $se_1$ and $te_1$ commute with
$\gamma$ and $\delta$. We have
$$s(g-g^{-1})te_1=(g^{-1}-g)ste_1 = (g-g^{-1})tse_1$$
which shows that $se_1$ and $\gamma$ commute.
Similar calculations show the remaining commutation relations.
This shows that the elements $\gamma$ and $\delta$ are
in $\hat A$, and we need to show that they generate $\hat A$. Note that 
$g-g^{-1}=$ $(g^2-1)g^{-1}$ is a generator of 
$J(k\langle g\rangle)/J(k\langle g\rangle)^2$; similarly for $h$.
Thus $(g-g^{-1})e_1$ and $(h-h^{-1})e_1$ generate the radical modulo 
the radical square of the image of $\OP$ in $B_1$, hence
these two elements together with $e_1$ generate the algebra $\OP e_1$. 
The two elements $\gamma$ and $\delta$ are obtained by multiplying 
$(g-g^{-1})e_1$ and $(h-h^{-1})e_1$ by 
$te_1$ and $se_1$, respectively, and the two elements $se_1$ and $te_1$ 
generate $S$ as an $\CO$-algebra.
It follows that the set $\{se_1, te_1, \gamma, \delta\}$ generates 
$B_1$ as an $\CO$-algebra. But then $\gamma$ and $\delta$
necessarily generate $\hat A$ as a unitary algebra. 

\medskip
We verify that $\gamma$ and $\delta$ satisfy the relations as stated. 
We have 
\begin{align*}
\gamma\delta&=(g-g^{-1})t(h-h^{-1})se_1=-(g-g^{-1})(h-h^{-1})tse_1\\
&=(h-h^{-1})(g-g^{-1})ste_1=-(h-h^{-1})s(g-g^{-1})te_1=-\delta\gamma
\end{align*}
whence the anti-commutation relation for $\gamma$ and $\delta$.
For the remaining relations, we first consider the element
$g-g^{-1}$ in $\CO\langle g\rangle$. 
This element acts on any $\CO\langle g\rangle$-module of rank one
as multiplication by $\zeta-\zeta^{-1}$ for some $p$-th root
of unity $\zeta$. This is an imaginary number; writing
$\zeta=$ $e^{\frac{2\pi m}{p}}$ for some
integer $m$, we get that $\zeta-\zeta^{-1}=$ 
$2\sin(\frac{2\pi m}{p})\tau$, where $\tau$ satisfies $\tau^2=-1$.
Thus $\frac{\tau}{2}(g-g^{-1})$ acts as multiplication by
$-\sin(\frac{2\pi m}{p})$. Since $T_p$ involves only odd powers of
$x$, it follows that 
$T_p(\frac{\tau}{2}(g-g^{-1}))$ acts as multiplication by
$\pm\sin(p\frac{2\pi m}{p})=$ $0$, and hence 
$T_p(\frac{\tau}{2}(g-g^{-1}))=$ $0$ in $\CO\langle g\rangle$,
or equivalently, $f_p(\tau(g-g^{-1}))=$ $0$.
We calculate the odd powers of $\gamma$ and $\delta$.
For $n=$ $2m+1$ for some integer $m\geq 0$ we have 
$$(te_1)^n= t(t^2e_1)^m=(-1)^mte_1= \tau^{n-1}te_1\ ,$$ 
and hence we have
$$\gamma^n= (g-g^{-1})^nt^ne_1 = \tau^{n-1}(g-g^{-1})^nte_1=
\tau^{-1}(\tau(g-g^{-1}))^nte_1$$
Thus, using again that $f_p$ involves only odd powers of $x$, we have
\[ f_p(\gamma) = \tau^{-1}f_p(\tau(g-g^{-1}))te_1=0. \] 
A similar calculation
yields $f_p(\delta)=0$. This shows that $\gamma$ and $\delta$
satisfy the relations as stated. That is, $\hat A$ is
a quotient of the unitary $\CO$-algebra 
\[ C=\CO\langle \gamma, \delta \ |\ \gamma\delta+\delta\gamma=0,
\ f_p(\gamma)=0=f_p(\delta)\rangle. \] 
As an $\CO$-module,
$\hat A$ is free of rank $p^2$. The relations defining $C$
imply that $C$ is generated, as an $\CO$-module, by the images
of the $p^2$ monomials $\gamma^i\delta^j$, with $0\leq$ $i, j\leq$ 
$p-1$, and hence $C$ is, as an $\CO$-module, a quotient of a
free $\CO$-module of rank $p^2$. This forces $C\cong$ $\hat A$,
whence the result. 
\end{proof}

If $B$ is a nilpotent block of some finite group algebra, then the 
largest $\CO$-free commutative algebra quotient of a basic algebra of 
$B$ is symmetric. Indeed, in that case the basic algebras of $B$ are 
isomorphic to $\OQ$ for some defect group $Q$ of $B$, and the largest 
$\CO$-free commutative algebra quotient of $\OQ$ is the symmetric 
$\CO$-algebra $\CO Q/Q'$, where $Q'$ is the derived subgroup of $Q$. 
In \cite[Remark 1.3]{KLN}, the question was raised whether this property 
characterises nilpotent blocks. Some evidence for this comes from a 
theorem of Okuyama and Tsushima in \cite{OkTs} which states that $B$ has 
a commutative (and necessarily symmetric) basic algebra if and only if 
$B$ is nilpotent with abelian defect groups.
For the sake of testing this question, we calculate the largest
commutative $\CO$-algebra quotient of the basic algebra $\hat A$
of the non-principal (and non-nilpotent) block $B_1$ of $\OG$, and show 
that this is indeed not symmetric.

The irreducible characters of $B_1$ have degree either $2$ or $4$. Thus 
the simple $K\tenO \hat A$-modules have dimension either $1$ or $2$. 
The number of simple $K\tenO \hat A$-modules of dimension $1$ is equal 
to $2p-1$, and this is also equal to the $\CO$-rank of the largest 
$\CO$-free commutative quotient of $\hat A$. This quotient is of the 
form $\hat A/I$, where $I$ is the smallest $\CO$-pure ideal in $\hat A$ 
which contains $[\hat A,\hat A]$. Its structure is as follows.

\begin{Proposition} \label{Ahatcommutators}
Let $\hat A =$ $\CO\langle \gamma, \delta \ |\ \gamma\delta+\delta\gamma=0,
\ f_p(\gamma)=0=f_p(\delta)\rangle$ as in the previous theorem.
\begin{enumerate}
\item[\rm (i)]
The set
$\{\gamma^i\delta^j\ |\ 1\leq i,j\leq p-1,\ i\ \mathrm{or}\ j
\ \mathrm{odd}\}$
is an $\CO$-basis of $[\hat A,\hat A]$. 
\item[\rm (ii)]
The smallest $\CO$-pure ideal in $\hat A$ which contains $[\hat A,\hat A]$
is equal to $\hat A\gamma\delta=\gamma\delta\hat A$, and
the set $\{\gamma^i\delta^j\ |\ 1\leq i,j\leq p-1\}$
is an $\CO$-basis of this ideal.
\item[\rm (iii)]
The largest $\CO$-free commutative quotient of $\hat A$ is isomorphic to
$$D = 
\CO\langle \mu, \nu\ |\ \mu\nu=\nu\mu=0,\ f_p(\mu)=f_p(\nu)=0\rangle\ .$$
The set $\{1\}\cup\{\mu^i, \nu^i\ |\ 1\leq i\leq p-1\}$ is an 
$\CO$-basis of $D$; in particular, the $\CO$-rank of this quotient is 
$2p-1$. 
\item[\rm (iv)]
We have $k\tenO D\cong k\langle \mu,\nu\ |\ \mu\nu=\nu\mu=0,\
\mu^p=\nu^p=0\rangle$, and the $k$-algebra $k\tenO D$ is not 
symmetric; in particular, the $\CO$-algebra $D$ is not symmetric.
\end{enumerate}
\end{Proposition}

\begin{proof} 
The algebra $K\tenO \hat A$ is split semisimple, and hence
$\rk_\CO(Z(A))+\rk_\CO([\hat A, \hat A])=\rk_\CO(\hat A)$.
Since $A$ is symmetric, we have $\dim_k(Z(A))+\dim_k([A,A])=$
$\dim_k(A)$. Since the canonical map $Z(\hat A)\to$ $Z(A)$ is
surjective, it follows that $\rk_\CO([\hat A,\hat A])=
\dim_k([A,A])$. Thus $[\hat A,\hat A]$ is an $\CO$-pure
$\CO$-submodule of $\hat A$.
The same arguments is in the proof of Lemma
\ref{commutatorbasis} show that the set
$\{\gamma^i\delta^j\ |\ 1\leq i,j\leq p-1,\ i\ \mathrm{or}\ j
\ \mathrm{odd}\}$ is contained in $[\hat A, \hat A]$. This
set spans an $\CO$-pure $\CO$-submodule of $\hat A$ mapping
onto $[A, A]$, and hence this set is an $\CO$-basis of
$[\hat A,\hat A]$. This proves (i). The ideal generated by
the set $\{\gamma^i\delta^j\ |\ 1\leq i,j\leq p-1,\ i\ \mathrm{or}\ j
\ \mathrm{odd}\}$ contains the set 
$\{\gamma^i\delta^j\ |\ 1\leq i,j\leq p-1\}$. The $\CO$-span of the
latter is an ideal, whence (ii). It follows from (ii) that 
$A/A\gamma\delta$ is the largest $\CO$-free commutative quotient of 
$\hat A$. 
The relations of this quotient are obtained from those of $\hat A$,
whence (iii). The image of the polynomial $f_p(u)$ in $k[u]$ is
$x^p$. Thus the relations of $k\tenO D$ follow from those of $D$.
The socle of the $k$-algebra $k\tenO D$ contains the images of
$\mu^{p-1}$ and $\nu^{p-1}$, hence has dimension at least $2$.
Since $k\tenO D$ is local, this shows that $k\tenO D$ is not 
symmetric, and hence neither is $D$.
\end{proof}

%


\bibliographystyle{amsplain}
\newcommand{\noopsort}[1]{}
\providecommand{\bysame}{\leavevmode\hbox to3em{\hrulefill}\thinspace}
\providecommand{\MR}{\relax\ifhmode\unskip\space\fi MR }
\providecommand{\MRhref}[2]{%
  \href{http://www.ams.org/mathscinet-getitem?mr=#1}{#2}
}
\providecommand{\href}[2]{#2}

\end{document}